\documentclass[a4paper,12pt]{article}
\usepackage{a4wide}
\usepackage{amsmath}
\usepackage{amssymb}
\usepackage{amsthm}
\usepackage{latexsym}
\usepackage{graphicx}
\usepackage[english]{babel}
\usepackage{makeidx}
              
\newtheorem{dfn} [subsection]{Definition}
\newtheorem{obs} [subsection]{Remark}
\newtheorem{exm} [subsection]{Example}

\newtheorem{prop}[subsection]{Proposition}

\newtheorem{teor}[subsection]{Theorem}
\newtheorem*{teor*}{Theorem}
\newtheorem{lema}[subsection]{Lemma}
\newtheorem{cor} [subsection]{Corollary}

\def\supp{\operatorname{supp}}

\def\bp{\operatorname{b-pol}}
\def\lcm{\operatorname{lcm}}
\def\gcd{\operatorname{gcd}}

\def\Mon{\operatorname{Mon}}

\def\depth{\operatorname{depth}}
\def\sdepth{\operatorname{sdepth}}

\usepackage{mathtools}

\DeclarePairedDelimiter\floor{\lfloor}{\rfloor}

\begin{document}
\selectlanguage{english}
\frenchspacing
\numberwithin{equation}{section}

\large
\begin{center}
\textbf{Polarization and spreading of monomial ideals}

Mircea Cimpoea\c s
\end{center}
\normalsize

\begin{abstract}
We characterize the monomial ideals $I\subset K[x_1,\ldots,x_n]$ with the property
that the polarization $I^p$ and $I^{\sigma^n}:=$ the ideal obtained from $I$ by 
the $n$-th iterated squarefree operator $\sigma$ are isomorphic via a permutation of 
variables. We give several methods to construct such ideals. We also
compare the depth and sdepth of $I$ and $I^{\sigma^n}$.



\end{abstract}

\section*{Introduction}

Let $K$ be a field and let $T=K[x_1,x_2,\ldots]$ be the ring of polynomials with countably indeterminates. For any integer $n\geq 1$, let $T_n:=K[x_1,\ldots,x_n]$ be
the polynomial ring in $n$ indeterminates. There is a well-known operator on monomial ideals, called \emph{polarization}, 
which preserve all the combinatorial and homological properties
of these ideals. On the other hand, in the shifting theory, introduced by Kalai \cite{kalai}, it is used another operator called 
the \emph{squarefree operator} (also known as the \emph{streching operator}, see \cite{ene}) which transforms an arbitrary monomial $u=x_{i_1}x_{i_2}\cdots x_{i_d}\in T$ into a squarefree monomial
$\sigma(u):=x_{i_1}x_{i_2+1}\cdots x_{i_d+d-1}\in T$.

Let $I\subset T_n$ be a monomial ideal with the minimal monomial set of generators $G(I)=\{u_1,\ldots,u_m\}$. One defines
$I^{\sigma}:=(\sigma(u_1),\ldots,\sigma(u_m))\subset T_{n+d-1}$,
where $d:=\deg(I)$ is the maximal degree of a monomial in $G(I)$. Note that $G(I^{\sigma})=\{\sigma(u_1),\ldots,\sigma(u_m)\}$.

Yanagawa \cite{yana} related the square-free operator $\sigma$ to an "alternative" polarization, defined by
$\bp(u):=x_{i_1,1}x_{i_2,2}\cdots x_{i_d,d}\in \widetilde T:=K[x_{i,j}\;:\;i,j\geq 1]$. Let $\bp(I)\subset \widetilde T$ be the
monomial with the minimal set of generators $G(\bp(I))=\{\bp(u_1),\ldots,\bp(u_m)\}$. In \cite[Theorem 3.4]{yana} and \cite[Corrolary 4.1]{okazaki} it was proved
that if $I$ is Borel fixed, then $\bp(I)$ is a polarization of $I$. Hence, $I^{\sigma}$ is a polarization of $I$, see \cite{okazaki} for further details.
In our paper, we discuss the relation between the squarefree operator $\sigma$ and polarization from another perspective.

In \cite{ene}, a monomial $x_{i_1}\cdots x_{i_d}\in T$ is called $t-spread$, if $i_j-i_{j-1}\geq t$ for $2\leq j\leq d$.
Also, a monomial ideal $I\subset T_n$ is called a $t$-spread monomial ideal, if it is generated
by t-spread monomials. In \cite[Corollary 1.7]{ene} it was proved that the $t$-fold iterated operator $\sigma^t$ establish a bijection between the monomials of $T$
and the $t$-spread monomials of $T$. In particular, given a monomial ideal $I\subset T_n$, the ideal $I^{\sigma^t}\subset T_{n+t(d-1)}$ is $t$-spread.

In general, the squarefree operator does not preserve any of the homological and combinatorial properties of an ideal $I\subset T_n$.
For instance, the ideal $I=(x_1^2,x_2^2)\subset T_2$ is 
a monomial complete intersection (c.i.), but $I^{\sigma}=(x_1x_2,x_2x_3)\subset T_3$ is not. However $I^{\sigma^t}=(x_1x_{t+1},x_2x_{t+2})$ is a c.i. for any $t\geq 2$.
In Proposition $1.2$ we prove that, in general, if $I\subset T_n$ is a monomial c.i. then $I^{\sigma^t}$ is a monomial c.i. for any $t\geq n$. This bound, as the
previous example shows, is sharp.

Given a monomial ideal $I\subset T_n$, we show in Proposition $1.4$ that for any $t\geq n$ there is a $K$-algebra monomorphism 
$\Phi:T_{nd}\rightarrow T_{td}$ which induce an injection 
$$\Phi|_{\{x_1,\ldots,x_{nd}\}}:\{x_1,\ldots,x_{nd}\} \rightarrow \{x_1,\ldots,x_{td}\},$$
such that $\Phi(I^{\sigma^n})=I^{\sigma^t}$. Consequently, in Corollary $1.5$ we show that $I^{\sigma^t}$, for $t\geq n$, are basically the same, from a homological
and combinatorial point of view. This result yields us to give a closer look to the ideal $I^{\sigma^n}$ and its connections with $I$. 
A problem which arise is to find classes of monomial ideals for which the homological and combinatorial properties are preserved when we switch from $I$ to $I^{\sigma^n}$.
Since the polarization preserves those properties, it is natural to ask when $I^p$ and $I^{\sigma^n}$ are essentially the same, i.e. they can be transformed one in another
through a permutation of variables.

Given a monomial $u=x_1^{a_1}\ldots x_n^{a_n}\in T_n$, we consider the \emph{polarization} of $u$ to be $u^p:=x_1x_{n+1}\cdots x_{(a_1-1)n+1} \cdots x_{n}x_{2n} \cdots x_{a_nn}$.
We say that a set of monomials $\mathcal M \subset T_n$ with $\deg(u)\leq d$, for any $u\in\mathcal M$,  is \emph{smoothly spreadable} if there exists a permutation $\tau$ on $\{1,2,\ldots,nd\}$ with
$\tau(j)-j \equiv 0 (\bmod\; n)$ such that the $K$-algebra isomorphism $\Phi:T_{nd}\rightarrow T_{nd}$, $\Phi(x_j):=x_{\tau(j)}$, has the property $\Phi(\sigma^n(u))=u^p$, for any $u\in\mathcal M$.
We say that a monomial ideal $I\subset T_n$ is smoothly spreadable if $G(I)$ is smoothly spreadable.

In Proposition $1.10$ we prove that if $I\subset T_n$ is a smoothly spreadable monomial ideal and if $v\in T_{n'}$ is a monomial with the support in 
$\{x_{n+1},\ldots,x_{n'}\}$, then the ideal $J=(I,v)\subset T_{n'}$ is smoothly spreadable. As a consequence, any monomial c.i. 
is smoothly spreadable, see Corollary $1.11$. In Proposition $1.13$ we prove that if $I=(x_1^{a_1}x_2^{b_1},\ldots,x_1^{a_m}x_2^{b_m})\subset T_2$ such that $a_1>a_2>\cdots>a_m\geq 0$, $0\leq b_1<b_2<\cdots<b_m$ and $a_1+b_1\leq a_2+b_2\leq \cdots \leq a_m+b_m$ then
$I$ is smoothly spreadable.

Let $I\subset T_n$ be a monomial ideal with $\deg(I)=d$. Let $1\leq k\leq n$, $1\leq j_1 < \cdots < j_k \leq n$ and $d_1,\ldots,d_k\geq 1$.
Let $J=(I,x_{j_1}^{d_1},\ldots,x_{j_k}^{d_k})\subset T_{n}$. In Proposition $1.14$, we show that if $d_1,\ldots,d_k$ satisfy certain condition and $I$ is smoothly spreadable,
then $J$ is smoothly spreadable. Given $\mathcal M=\{u_1,\ldots,u_m\}\in T_n$ a set of monomial with $\deg(u_i)=d, (\forall)1\leq i\leq m$, and $\mathcal M'=\{v_1,\ldots,v_m\}\in T_{n'}$
a set of monomials with $\supp(v_i)\subset \{x_{n+1},\ldots,x_{n'}\}$, $(\forall)1\leq i\leq m$, such that $\mathcal M$ and $\mathcal M'$ are smoothly spreadable, we show that
the set $\{u_iv_{\ell}\;:\;1\leq i\leq m,\;1\leq \ell\leq m' \}$ is smoothly spreadable, see Proposition $1.15$. 

The main result of the first section is Theorem $1.18$, where we give a combinatorial characterization for smoothly spreadable ideals.
Let $\mathcal M=\{u_1,\ldots,u_m\}\subset T_n$ be a set of monomials, $u_i:=\prod_{j=1}^n x_j^{a_{i,j}}, \; 1\leq i\leq m$. We prove that
$\mathcal M$ is smoothly spreadable if and only if for any $1\leq i < \ell \leq m$ and $1\leq j\leq n$ we have that 
\footnotesize
\begin{eqnarray*}
\min\{a_{i,j},a_{\ell, j}\} = |\{(a_{i,1}+\cdots+a_{i,j-1})n+j,(a_{i,1}+\cdots+a_{i,j-1}+1)n+j,\ldots,(a_{i,1}+\cdots+a_{i,j}-1)n+j\} \\
\cap\{(a_{\ell,1}+\cdots+a_{\ell,j-1})n+j,(a_{\ell,1}+\cdots+a_{\ell ,j-1}+1)n+j,\ldots,(a_{\ell, 1}+\cdots+a_{\ell ,j-1}+a_{\ell,j}-1)n+j\}|.
\end{eqnarray*}
\normalsize
We also provide several examples, see Example $1.8$, $1.12$, $1.16$ and $1.19$.

Given $I\subset T_n$ a monomial ideal, the lcm-lattice of $I$ encodes several combinatorial and homological properties of $I$,
see for instance \cite{gash}, \cite{map}, \cite{ichim2} and \cite{ichim3}. In the second section we discuss the connection between
 lcm-lattices of $I$ and $I^{\sigma^n}$. Using the main results from \cite{ichim2} we prove in Theorem $2.5$ that:
$\depth(T_{nd}/I^{\sigma^n})\leq \depth(T_n/I)+n(d-1)$,
$\sdepth(T_{nd}/I^{\sigma^n})\leq \sdepth(T_n/I)+n(d-1)$ and
$\sdepth(I^{\sigma^n})\leq \sdepth(I)+n(d-1)$. We note that the equalities hold if $I$ is smoothly spreadable, see Remark $2.2$.

\section{Main results}

Let $T:=K[x_1,x_2,\ldots]$ be the ring of polynomials with countably indeterminates and let $\Mon(T)$ be the set of monomials of $T$.
Let $u\in \Mon(T)$ and write $u=x_{i_1}x_{i_2}\cdots x_{i_d}$ with $i_1\leq i_2\leq \cdots \leq i_d$.
$u$ is called \emph{t-spread}, see \cite{ene}, if $i_j-i_{j-1}\geq t$ for all $2\leq j\leq d$.
Note that any monomial is $0$-spread, while the squarefree monomials are $1$-spread.
The \emph{squarefree operator}, $\sigma:\Mon(T)\rightarrow \Mon(T)$, is defined by $$\sigma(u):=x_{i_1}x_{i_2+1}\cdots x_{i_d+d-1},$$
see \cite[Page 214]{hehi} and \cite[Definition 1.5]{ene}.

Given an integer $t\geq 0$, let $\Mon(T,t)$ be the set of $t$-spread monomials. According to \cite[Lemma 1.6]{ene}, the restriction map 
$\sigma:\Mon(T,t) \rightarrow \Mon(T,t+1)$ is bijective. Consequently, the iterated map $\sigma^t :\Mon(T)\rightarrow \Mon(T,t)$ is also bijective.

For any integer $n\geq 1$, we denote $T_n:=K[x_1,\ldots,x_n]$, the ring of polynomials in $n$ indeterminates.
Let $0\neq I\subset T_n$ be a monomial ideal. We denote $G(I)$ the set of minimal monomial generators of $I$ and
$\deg(I):=\max\{\deg(u)\;:\;u\in G(I)\}$,
the maximal degree of a minimal monomial generator of $I$.

The ideal $I^{\sigma}\subset T_{n+d-1}$ is the ideal generated by $\sigma(u)$ with $u\in G(I)$, see \cite[Definition 1.8]{ene}.
Equivalently, we have that $G(I^{\sigma})=\{\sigma(u)\;:\;u\in G(I)\}$. 

For any $t\geq 1$, $u\in T_n$ a monomial and $I\subset T_n$ a monomial ideal, we define recursively
$$\sigma^t(u):=\sigma(\sigma^{t-1}(u)),\; I^{\sigma^t}:=(I^{\sigma^{t-1}})^{\sigma}.$$
In the following, for $t\geq n$, we will consider $I^{\sigma^t}$ as an ideal in $T_{td}$, where $d=\deg(I)$.

\begin{lema}
Let $u,v\in T_n$ be two monomials and let $t\geq n$ be an integer. If $\gcd(u,v)=1$ then 
$\gcd(\sigma^t(u),\sigma^t(v))=1$.
\end{lema}

\begin{proof}
We write $u=x_{i_1}\cdots x_{i_d}$, $v=x_{j_1}\cdots x_{j_{d'}}$, such that $1\leq i_1\leq \cdots \leq i_d\leq n$ and $1\leq j_1\leq \cdots \leq j_{d'}\leq n$.
It follows that $$\sigma^t(u) = x_{i_1}x_{i_2+t}\cdots x_{i_{d}+(d-1)t},\; \sigma^t(v) = x_{j_1}x_{j_2+t}\cdots x_{j_{d'}+(d'-1)t}.$$
By our assumptions, it follows that 
\begin{eqnarray*}
& 1+(\ell-1)t \leq i_{\ell}+(\ell-1)t \leq \ell t-1,\;(\forall)1\leq \ell\leq d,\\
& 1+(\ell-1)t \leq j_{\ell}+(\ell-1)t \leq \ell t-1,\;(\forall)1\leq \ell\leq d'.
\end{eqnarray*}
Hence $i_{\ell}+(\ell-1)t = j_{\ell'}+(\ell'-1)t$ implies $\ell=\ell'$ and $i_{\ell}=j_{\ell}$. This completes the proof.
\end{proof}

\begin{prop}
 Let $I\subset T_n$ be a complete intersection (c.i.) monomial ideal with $\deg(I)=d$. Then $I^{\sigma^t}\subset T_{td}$ is a complete intersection, for any $t\geq n$.
\end{prop}

\begin{proof}
Assume that $G(I)=\{u_1,\ldots,u_m\}$. If $m=1$ then there is nothing to prove. Assume $m\geq 2$.
It is well known that $I$ is a complete intersection if and only if $\gcd(u_i,u_j)=1$, $(\forall)i\neq j$.
The conclusion follows from Lemma $1.1$.
\end{proof}

\begin{obs}
\emph{
Let $n\geq s\geq 1$ be two integers.
Let $\alpha: a_1<a_2<\cdots<a_s$, $\beta: b_1<b_2<\cdots<b_s$ be two sequences of positive integers such that $a_i\leq b_i$ for all $1\leq i\leq s$ and $b_s = n$.
In \cite{mircea}, we studied monomial ideals of the form
$$I_{\alpha,\beta}=(x_{a_1}x_{a_1+1}\cdots x_{b_1},\ldots,x_{a_s}x_{a_s+1}\cdots x_{b_s})\subset T_n.$$
Note that $I_{\alpha,\beta}=J^{\sigma},\;\text{ where } J= (x_{a_1}^{b_1-a_1},\ldots,x_{a_s}^{b_s-a_s})$. 
Also, $J$ and $J^{\sigma^t}$, $t\geq n$, are monomial c.i. but, in general, $I_{\alpha,\beta}$ is not.
Conversely, given $J=(x_{a_1}^{c_1},\ldots,x_{a_s}^{c_s}) \subset T_{a_s}$ such that $a_i+c_i<a_{i+1}+c_{i+1}$, $(\forall)1\leq i\leq s-1$,
then $J^{\sigma}=I_{\alpha,\beta}$, where $b_i=a_i+c_i,(\forall)1\leq i\leq s$.
}
\end{obs}

\begin{prop}
Let $I\subset T_n$ be a monomial ideal. Let $d=\deg(I)$.
Let $t\geq n$ be an integer. We consider the $K$-algebra map defined by
$$\Phi_t : T_{nd}\rightarrow T_{td},\;\Phi_t(x_j):=x_{\varphi(j,t)},\;(\forall)1\leq j\leq nd, \text{ where }\varphi(j,t):=\floor*{\frac{j-1}{n}}\cdot (t-n) +j.$$
It holds that $I^{\sigma^t} = \Phi_t(I^{\sigma^n})$, as ideals in $T_{td}$.
\end{prop}

\begin{proof}
Assume $G(I)=\{u_1,\ldots,u_m\}$.
As in the proof of Lemma $1.1$, we can write
$$\sigma^n(u_j) = \prod_{\ell=1}^{\deg(u_j)} x_{\alpha(j,\ell)} \in T_{nd}, \text{ where } 1+(\ell-1)n\leq \alpha(j,\ell) \leq \ell n,\;(\forall)1\leq j\leq m. $$
It follows that
$$\sigma^t(u_j)=\sigma^{t-n}(\sigma^n(u_j))=\prod_{\ell=1}^{\deg(u_j)} x_{\varphi(\alpha(j,\ell),t)} \in T_{td},$$
where $\varphi$ was defined above.
\end{proof}

The following Corollary shows that the ideals $I^{\sigma^t}$, $t\geq n$, are essentially determinated by $I^{\sigma^n}$.
Therefore, it is interesting to study the mapping $I\mapsto I^{\sigma^n}$.

\begin{cor}
 We have that
\begin{enumerate}
 \item[(1)] $T_{nd}/I^{\sigma^n} \cong T_{td}/(I^{\sigma^t},x_{n+1},\ldots,x_t,x_{t+n+1},\ldots,x_{2t},\ldots,x_{n+1+(t-1)d},\ldots,x_{td}).$
 \item[(2)] $\depth(T_{nd}/I^{\sigma^n})=\depth(T_{td}/I^{\sigma^t})-(n-t)d$.
 \item[(3)] $\sdepth(T_{nd}/I^{\sigma^n})=\sdepth(T_{td}/I^{\sigma^t})-(n-t)d$.
 \item[(4)] $\sdepth(I^{\sigma^n})=\sdepth(I^{\sigma^t})-(n-t)d$.
\end{enumerate}
\end{cor}

\begin{proof}
(1) follows from Proposition $1.4$. (2) is a direct consequence of (1).

(3) From Proposition $1.4$ it follows that
$$T_{td}/I^{\sigma^t} \cong (T_{nd}/I^{\sigma^n})[x_{n+1},\ldots,x_t,x_{t+n+1},\ldots,x_{2t},\ldots,x_{n+1+(t-1)d},\ldots,x_{td}],$$
hence the conclusion follows from \cite[Lemma 3.6]{hvz}.

(4) From Proposition $1.4$ it follows that 
$$I^{\sigma^t} = \Phi(I^{\sigma^n})T_{td} \cong  I^{\sigma^n}T_{td},$$ 
hence the conclusion follows from \cite[Lemma 3.6]{hvz}.
\end{proof}

Given a monomial $u\in T_n$, $u=x_1^{a_1}x_2^{a_2}\cdots x_n^{a_n}$ with $\deg(u)=d$, we define 
$$u^p:= \sigma^n(x_1^{a_1})\sigma^n(x_2^{a_2})\cdots \sigma^n(x_n^{a_n}) = $$
$$ = x_1x_{n+1}\cdots x_{(a_1-1)n+1}x_2x_{n+2}\cdots x_{(a_2-1)n+2}\cdots x_nx_{2n}\cdots x_{a_n n}\in T_{td}.$$

Let $I\subset T_n$ be a monomial ideal with $\deg(I)=d$. . The \emph{polarization} of $I$ is the ideal
$$I^p:=(u_1^p,\ldots,u_m^p)\subset T_{nd}.$$
We note that the above definition is equivalent to the standard definition of polarization.

\begin{obs}
\emph{
Proposition $1.4$ and Corollary $1.5$ show that, given a monomial ideal $I\subset T_n$ with $\deg(I)=d$, the operator $\sigma$ could produce, by iteration, essentially different ideals up to the $n$-th step and,
also, that it is important to study the relations between $I$ and $I^{\sigma^n}$. From Proposition $1.2$, if $I\subset T_n$ is a monomial complete intersection (c.i.) then 
$I^{\sigma^n}\subset T_{nd}$ is also c.i. However, the converse is not true. For instance $I=(x_1^2x_2,x_2^2)\subset T_2$ is not a c.i. but $I^{\sigma^2}=(x_1x_3x_6,x_2x_4)\subset T_6$ is.
Since $I^{\sigma^n}$ is a square-free monomial ideal, a question which arises naturally is to characterize the monomial 
ideals $I\subset T_n$ with the property that $I^{\sigma^n}$ and $I^{p}$ are essentially the same, i.e. they are isomorphic via a permutation of variables. 
}
\end{obs}

\begin{dfn}
Let $d,m>0$ be two integers and let $\mathcal M=\{u_1,\ldots,u_m\}\subset T_n$ be a set of monomials with $\deg(u_i)\leq d$, $(\forall)1\leq i\leq m$.
We say that $\mathcal M$ is \emph{smoothly spreadable} if there exists a permutation 
$$\tau:\{1,2,\ldots,nd\} \rightarrow \{1,2,\ldots,nd\} \text{ with } \tau(j)-j \equiv 0 (\bmod \;n), \;(\forall)1\leq j\leq nd,$$
 such that the $K$-algebra isomorphism
 $$\Phi:T_{nd}\rightarrow T_{nd},\;\Phi(x_j):=x_{\tau(j)},\;(\forall)1\leq j\leq nd,$$
satisfies $\Phi(\sigma^n(u_i)) = u_i^p$, for all $1\leq i\leq m$. 

Let $I\subset T_n$ be a monomial ideal with $\deg(I)=d$. We say that $I$ is \emph{smoothly spreadable} if the set $G(I)$ is smoothly spreadable.
In particular, we have $\Phi(I^{\sigma^n})=I^p$.
\end{dfn}

\begin{exm}
\emph{
(1) We consider the ideal $I=(x_1x_2x_3,\; x_2^2 x_3)\subset T_3$. We have $\deg(I)=3$, so
$I^{\sigma^3}=(x_1x_5x_9,\; x_2x_5x_9)\subset T_9 \text{ and } I^p=(x_1x_2x_3,\; x_2x_5x_3)\subset T_9. $
We define the permutation
$\tau:\{1,2,\ldots,9\} \rightarrow \{1,2,\ldots,9\}$ by
$\tau(1)=1$, $\tau(4)=4$, $\tau(7)=7$, $\tau(2)=5$, $\tau(5)=2$, $\tau(8)=8$, $\tau(3)=6$, $\tau(6)=9$ and $\tau(9)=3$. Let
$$\Phi:T_9 \rightarrow T_9,\; \Phi(x_j):=x_{\tau(j)}, (\forall)1\leq j\leq 9.$$
 It is easy to see that $\Phi(x_1 x_5 x_9)=x_1 x_2 x_3$ and $\Phi(x_2x_5x_9)=x_2x_5x_3$,
hence, according to Definition $1.7$, the ideal $I$ is smoothly spreadable.}

\emph{
(2) We consider the ideal $I=(x_1^3x_2x_3^2,\; x_1x_2^2x_3^3)\subset T_3$. We have $\deg(I)=6$, thus
\small
\begin{eqnarray*}
& I^{\sigma^3}=(x_1x_4x_7x_{11}x_{15}x_{18},\; x_1x_5x_8x_{12}x_{15}x_{18})\subset T_{18},
I^p = (x_1x_4x_7x_{11}x_2x_3x_6,\; x_1x_2x_5x_3x_6x_9)\subset T_{18}, \\
& \gcd(x_1x_4x_7x_{11}x_{15}x_{18},\; x_1x_5x_8x_{12}x_{15}x_{18})=x_1x_{15}x_{18},\\
& \gcd(x_1x_4x_7x_{11}x_2x_3x_6,\; x_1x_2x_5x_3x_6x_9)=x_1x_2x_3x_6.
\end{eqnarray*}
\normalsize
Hence $I$ is not smoothly spreadable.}
\end{exm}

\begin{prop}
Let $I\subset T_n$ be a monomial ideal with $\deg(I)=d$. If $I$ is smoothly spreadable, then for any $t\geq n$ we have
\begin{enumerate}
\item[(1)] $\depth(T_{td}/I^{\sigma^t})=\depth(T_n/I)+td-n$.
\item[(2)] $\sdepth(T_{td}/I^{\sigma^t})=\sdepth(T_n/I)+td-n$.
\item[(3)] $\sdepth(I^{\sigma^t})=\sdepth(I)+td-n$.
\end{enumerate}
\end{prop}

\begin{proof}
(1) Since $I$ is smoothly spreadable, according to Definition $1.7$, it follows that $T_{nd}/I^p \cong T_{nd}/I^{\sigma^n}$.
 The conclusion follows from Corrolary $1.5$ and the well known properties of polarization.

(2) According to \cite[Corollary 4.4]{ichim0}, we have $\sdepth(T_{nd}/I^p)=\sdepth(T_n/I)+nd-n$. On the other hand, 
$T_{nd}/I^p \cong T_{nd}/I^{\sigma^n}$, hence the conclusion follows from Corollary $1.5$.

(3) According to \cite[Corollary 4.4]{ichim0}, we have $\sdepth(I^p)=\sdepth(I)+nd-n$. On the other hand, 
$I^p \cong I^{\sigma^n}$, hence the conclusion follows from Corollary $1.5$.
\end{proof}

For any monomial $u\in T$, the \emph{support} of $u$ is $\supp(u):=\{x_j\;:\; x_j|u\}$. In the following, we present some constructive methods to produce 
smoothly spreadable monomial ideals.

\begin{prop}
Let $1\leq n < n'$ be two integers. Let $I\subset T_n$ be a monomial ideal and let $1\neq v\in T_{n'}$ be a monomial with $\supp(v)\subset\{x_{n+1},\ldots,x_{n'}\}$.
Let $J=(I,v)\subset T_{n'}$. Then 
$I$  is smoothly spreadable if and only if $J$ is smoothly spreadable.
\end{prop}

\begin{proof}
Assume that $I$ is smoothly spreadable. According to Definition $1.7$, there exists a permutation $\tau:\{1,2,\ldots,nd\} \rightarrow \{1,2,\ldots,nd\}$, 
with $\tau(j)-j \equiv 0 (\bmod\; n)$, $(\forall)1\leq j\leq nd$, such that the $K$-algebra isomorphism
 $$\Phi:T_{nd}\rightarrow T_{nd},\;\Phi(x_j):=x_{\tau(j)},\;(\forall)1\leq j\leq nd,$$
satisfies $\Phi(\sigma^n(u_i)) = u_i^p$, for all $1\leq i\leq m$. 
If $v=x_{n+1}^{a_{n+1}}\cdots x_{n'}^{a_{n'}}$, then
\begin{eqnarray*}
& \sigma^{n'}(v)=x_{n+1}x_{n'+(n+1)}\cdots x_{(a_{n+1}-1)n'+(n+1)}\cdots x_{(a_{n+1}+\cdots+a_{n'-1}+1)n'}\cdots 
x_{(a_{n+1}+\cdots+a_{n'-1}+a_n)n'},\\
& \text{ and } v^{p'}=x_{n+1}x_{n'+(n+1)}\cdots x_{(a_{n+1}-1)n'+(n+1)}\cdots x_{n'}x_{2n'}\cdots x_{a_{n'}n'},
\end{eqnarray*}
where $()^{p'}$ is the polarization in $T_{n'}$.
For any $1\leq j\leq n'-n$, we define the bijection
\small
\begin{eqnarray*}
& \sigma_j:\{n+j,n'+n+j,\ldots,(d-1)n'+n+j\} \rightarrow \{n+j,n'+n+j,\ldots,(d-1)n'+n+j\},\\
% such that
& \sigma_j((a_{n+1}+\cdots+a_{n+j-1})n'+n+j)=n+j,\;\\
& \sigma_j((a_{n+1}+\cdots+a_{n+j-1}+1)n'+n+j)=n'+n+j,\ldots,\\
& \sigma_j ((a_{n+1}+\cdots+a_{n+j}-1)n'+n+j) = (a_{n+j}-1)n'+n+j , 
\end{eqnarray*}
\normalsize
which satisfies the property
\small
\begin{eqnarray*}
& \sigma_j(\{n+j,\ldots,(a_{n+1}+\cdots+a_{n+j-1}-1)n'+n+j,(a_{n+1}+\cdots+a_{n+j})n'+n+j,\ldots\\ 
& ,(d-1)n'+n+j\})=  \{a_{n+j}n'+n+j,(a_{n+j}+1)n'+n+j,\ldots,(d-1)n'+n+j\}.
\end{eqnarray*}
\normalsize
We define the permutation
\small{
$$\tau':\{1,\ldots,n'd'\}\rightarrow \{1,\ldots,n'd'\},\; \tau'(j)= \begin{cases} \varphi(\tau(\psi(j,n')),n),\;& j\equiv 1,\ldots,n (\bmod n'),j\leq n'd\\ 
\sigma_{j-n}(j),\;& j\equiv n+1,\ldots,n' (\bmod n')\\
j,\;& \text{otherwise} \end{cases},$$}
where $\psi(j,n'):=\floor*{\frac{j-1}{n'}}\cdot (n-n') +j$ and $\varphi(j,n):=\floor*{\frac{j-1}{n}}\cdot (n'-n) +j$. 
We let
$$\Phi':T_{n'd'}\rightarrow T_{n'd'},\; \Phi'(x_j):=x_{\tau'(j)}, (\forall)1\leq j\leq n'd'.$$
From the definition of $\sigma_j$ and $\Phi'$, it follows by straightforward computations that $\Phi'(\sigma^{n'}(v))=v^{p'}$.
Let $u=x_1^{a_1}\cdots x_n^{a_n}\in G(I)$. We have that
\begin{eqnarray*}
& u^{\sigma^{n'}}=x_1x_{n'+1}\cdots x_{(a_1-1)n'+1} \cdots x_{(a_1+\cdots+a_{n-1})n'+n}\cdots x_{(a_1+\cdots +a_n-1)n'+n}\\
& \text{ and } u^{p'}=x_1x_{n'+1}\cdots x_{(a_1-1)n'+1} \cdots x_nx_{n'+n}\cdots x_{(a_n-1)n'+n}.
\end{eqnarray*}
We claim that $\Phi'(\sigma^{n'}(u))=u^{p'}$. Indeed, we have that 
\begin{eqnarray*}
\Phi'(\sigma^{n'}(u)) = x_{\tau'(1)}x_{\tau'(n'+1)}\cdots x_{\tau'((a_1-1)n'+1)} \cdots 
x_{\tau'((a_1+\cdots +a_{n-1})n'+n)}\cdots x_{\tau'((a_1+\cdots +a_n-1)n'+n)} = \\
 = x_{\varphi(\tau(1),n)}x_{\varphi(\tau(n+1),n)}\cdots x_{\varphi(\tau((a_1-1)n+1),n)} \cdots 
x_{\varphi(\tau((a_1+\cdots +a_{n-1}+1)n),n)} \cdots x_{\varphi(\tau(a_1+\cdots +a_n)n,n)} = \\
 x_{\varphi(1,n)}x_{\varphi(n+1,n)} \cdots x_{\varphi((a_1-1)n+1,n)}\cdots x_{\varphi(n,n)} \cdots x_{\varphi(a_nn,n)} = 
x_1\cdots x_{(a_1-1)n'+1}\cdots x_{n'}\cdots x_{a_nn'}=u^{p'}, 
\end{eqnarray*}
as required.
Now, assume that $J$ is smoothly spreadable. Let 
$$\Phi':T_{n'd'}\rightarrow T_{n'd'}, \; \Phi'(x_j)=x_{\tau'(j)},\; (\forall)1\leq j\leq n'd',$$ be the $K$-algebra isomorphism, as in Definition $1.7$, 
such that $\Phi'(\sigma^{n'}(u))=u^{p'}$, $(\forall)u\in G(I)$. We define
$$\tau:\{1,\ldots,nd\}\rightarrow \{1,\ldots,nd\},\; \tau(j):=\psi(\tau'(\varphi(j,n)),n'),$$
and $\Phi:T_{nd}\rightarrow T_{nd}$, $\Phi(x_j):=x_{\tau(j)}$, $1\leq j\leq nd$. Let $u\in G(I)$. Since $\Phi'(u)=u^{p'}$, 
by straightforward computations, we get $\Phi(u)=u^p$.
\end{proof}

We recall the fact that a monomial ideal $I\subset T_n$ with $G(I)=\{v_1,\ldots,v_m\}$ is a complete intersection (c.i.) if and only if 
$\supp(v_i)\cap \supp(v_j)=\emptyset$, $(\forall)i\neq j$.

\begin{cor}
If $I\subset T_n$ is a monomial c.i. then $I$ is smoothly spreadable. 
\end{cor}

\begin{proof}
It follows from Proposition $1.10$, by induction on the number of minimal monomial generators of $I$.
\end{proof}

\begin{exm}\emph{
(1) Let $I=(x_1^3,\; x_2x_3)\subset T_3$. Note that $I$ is a complete intersection, but $I^{\sigma}=(x_1x_2x_3,\; x_2x_4)\subset T_6$ and $I^{\sigma^2}=(x_1x_3x_5,\; x_2x_5)\subset T_6$ are not.
On the other hand, for any $t\geq 3$, $I^{\sigma^t} = (x_1x_{t+1}x_{2t+1},\; x_2x_{3+t})\subset T_{3t}$ is a complete intersection.}

\emph{
(2) Let $I=(x_1^2x_2,\; x_2^2)\subset T_2$. We have that $I^{\sigma^t}=(x_1x_{1+t}x_{2+2t},\; x_2x_{2+t})$, $(\forall)t\geq 1$, hence $I,I^{\sigma}$ are not complete intersection,
but $I^{\sigma^t}$ is a complete intersection for any $t\geq 2$. This shows indirectly that $I$ is not smoothly spreadable.}
\end{exm}

In the following proposition we characterize the monomial ideals in $T_2$ which are smoothly spreadable.

\begin{prop}
Let $a_1>a_2>\cdots >a_m \geq 0$ and $0 \leq b_1<b_2<\cdots<b_m$ be two sequences of integers, where $m\geq 1$, and 
let $I:=(x_1^{a_1}x_2^{b_1},\ldots,x_1^{a_m}x_2^{b_m})\subset T_2$. We have that:
\begin{enumerate}
 \item[(1)] If $a_1+b_1\leq a_2+b_2 \leq \cdots \leq a_m+b_m:=d>0$ then $I$ is smoothly spreadable.
 \item[(2)] If $I$ is smoothly spreadable, then $a_2+b_2\leq \cdots \leq a_{m-1}+b_{m-1}$. Moreover, if $b_1>0$ then
            $a_1+b_1\leq a_2+b_2$ and if $a_m>0$ then $a_{m-1}+b_{m-1}\leq a_m+b_m$.
\end{enumerate}
\end{prop}

\begin{proof}
$(1)$ Let $u_i:=x_1^{a_i}x_2^{b_i}$, $1\leq i\leq m$. Note that $G(I)=\{u_1,\ldots,u_m\}$. 
For $1\leq i\leq m$ we have
\begin{equation}\label{p11-0}
\sigma^2(u_i)=x_1x_3\cdots x_{2a_i-1}x_{2(a_i+1)}x_{2(a_i+2)} \cdots x_{2(a_i+b_i)} ,\;
u_i^p = x_1x_3\cdots x_{2a_i-1}x_2x_4\cdots x_{2b_i}.
\end{equation}
We define the map $\lambda:\{1,2,\ldots,d\}\rightarrow \{1,2,\ldots,d\}$ as follows:
\begin{eqnarray*}
\lambda(j)=b_m + j, & 1\leq j\leq a_m \\
\lambda(a_1+j)=j, & 1\leq j\leq b_1 \\
\lambda(a_i+j) = b_{i-1} + j, & 1\leq j\leq a_{i-1}-a_i,\;2\leq i\leq m \\
\lambda(a_i+b_i+j) = a_i+b_{i} - a_{i+1}+j, & 1\leq j\leq b_{i+1}-b_i,\;1\leq i\leq m-1. 
\end{eqnarray*}
From the above relations it follows that
\small
\begin{equation}\label{p11-1}
\lambda(\{a_i+1,a_i+2,\ldots,a_{i-1},a_{i-1}+b_{i-1}+1,a_{i-1}+b_{i-1}+2,\ldots,a_{i-1}+b_i\}) = 
 \{b_{i-1}+1,b_{i-1}+2,\ldots,b_i\},
\end{equation}
\normalsize
 for any $2\leq i\leq m$, where $b_0=0$. Also, we have that
\begin{equation}\label{p11-2}
\lambda(\{a_1,a_1+1,\ldots,a_1+b_1\})=\{1,2,\ldots,b_1\},
\end{equation}
\begin{equation}\label{p11-22}
\lambda(\{1,2,\ldots,a_m\})=\{b_m+1,\ldots,b_m+a_m\}.
\end{equation}
We prove by induction on $i\geq 1$ that 
\begin{equation}\label{p11-3}
\lambda(\{a_i+1,a_i+2,\ldots,a_i+b_i\})=\{1,2,\ldots,b_i\}.
\end{equation}
From \eqref{p11-2}, the above assertion holds for $i=1$. Assume that $i>1$ and \eqref{p11-3} holds for $i-1$, that is
\begin{equation}\label{p11-4}
\lambda(\{a_{i-1}+1,a_{i-1}+2,\ldots,a_{i-1}+b_{i-1}\})=\{1,2,\ldots,b_{i-1}\}.
\end{equation}
From \eqref{p11-1} and \eqref{p11-4} it follows that \eqref{p11-3} holds for $i$, hence the induction is complete.
In particular, it follows that
\begin{equation}\label{p11-5}
 \lambda(\{a_m+1,a_m+2,\ldots,a_m+b_m\})=\{1,2,\ldots,b_m\}
\end{equation}
From \eqref{p11-22} and \eqref{p11-5} it follows that 
% On the other hand
% \begin{equation}\label{p11-3}
% \bigcup_{i=2}^m \{a_i+1,a_i+2,\ldots,a_{i-1},a_{i-1}+b_{i-1}+1,a_{i-1}+b_{i-1}+2,\ldots,a_{i-1}+b_i\} = \{a_m+1,a_m+2,\ldots,a_1+b_1-1\}.
% \end{equation}
$\lambda$ is a permutation on $\{1,\ldots,d\}$.
We consider the $K$-algebra isomorphism 
$$\Phi:T_{2d} \rightarrow T_{2d}, \text{ where } \Phi(x_{j})=\begin{cases} x_j,& j=2k+1 \\ x_{2\lambda(k)}, & j=2k \end{cases}.$$
From \eqref{p11-3} it follows that 
$$\Phi(x_{2(a_i+1)}x_{2(a_i+2)} \cdots x_{2(a_i+b_i)}) = x_2x_4\cdots x_{2b_i},$$
hence \eqref{p11-0} implies $\Phi(\sigma^2(u_i))=u_i^p$, $(\forall)1\leq i\leq m$, as required.

$(2)$ Without any loss of generality, we can assume $b_1>0$ and $a_m>0$.
Suppose that there exists $1\leq \ell\leq m-1$ such that $a_{\ell}+b_{\ell}>a_{\ell+1}+b_{\ell+1}$.

Since $I$ is smoothly spreadable,
according to Definition $1.7$, there exists a permutation 
$$\tau:\{1,2,\ldots,2d\}\rightarrow \{1,2,\ldots,2d\} \text{ with } \tau(j)\equiv j (\bmod\; 2),(\forall)1\leq j\leq 2d,$$
such that the induced $K$-algebra isomorphism 
$$\Phi:T_{2d}\rightarrow T_{2d},\;\Phi(x_j):=x_{\tau(j)},\;(\forall)1\leq j\leq 2d,$$
satisfies $\Phi(\sigma^2(u_i))=u_i^p$, $(\forall)1\leq i\leq m$. In particular, from \eqref{p11-0} we have that \small
\begin{eqnarray*}
& \Phi(\sigma^2(u_{\ell}))=\Phi(x_1x_3\cdots x_{2a_{\ell}-1}x_{2(a_{\ell}+1)}x_{2(a_{\ell}+2)} \cdots x_{2(a_{\ell}+b_{\ell})}) = 
x_1x_3\cdots x_{2a_{\ell}-1}x_2x_4 \cdots x_{2b_{\ell}} \\ 
& \Phi(\sigma^2(u_{\ell+1}))=\Phi(x_1x_3\cdots x_{2a_{\ell+1}-1}x_{2(a_{\ell+1}+1)} \cdots x_{2(a_{\ell+1}+b_{\ell+1})}) = 
x_1x_3\cdots x_{2a_{\ell+1}-1}x_2x_4 \cdots x_{2b_{\ell+1}}, 
\end{eqnarray*}
\normalsize
which is equivalent to
\begin{equation}\label{p11-6}
 \tau(\{2a_{\ell+j}+2,2a_{\ell+j}+4,\ldots,2(a_{\ell+j}+b_{\ell+j})\})=\{2,4,\ldots,b_{\ell+j}\},\;j=0,1.
\end{equation}
Since $a_{\ell}>a_{\ell+1}$ and $a_{\ell}+b_{\ell}>a_{\ell+1}+b_{\ell+1}$, \eqref{p11-6} yields to a contradiction.
\end{proof}

The following proposition provides a method do construct smoothly spreadable ideals, by adding new monomials 
which are powers of variables.

\begin{prop}
Let $I\subset T_n$ be a monomial ideal with $\deg(I)=d$. Let $1\leq k\leq n$, $1\leq j_1 < \cdots < j_k \leq n$ and $d_1,\ldots,d_k\geq 1$.
Let $J=(I,x_{j_1}^{d_1},\ldots,x_{j_k}^{d_k})\subset T_{n}$. 
We assume that $G(J)=G(I)\cup \{x_{j_1}^{d_1},\ldots,x_{j_k}^{d_k}\}$. 
For any $u\in T_n$ and $1\leq \ell\leq k$, let $u^{(\ell)}$ be the monomial obtained from 
$u$ by replacing $x_{j_{\ell}+1},\ldots,x_n$ with $1$.
The following holds:
\begin{enumerate}
 \item[(1)] If $d_{\ell}\geq \max\{deg(u^{(\ell)})\;:\;u\in G(I),\;x_{j_{\ell}}|u\}$ for all $1\leq \ell \leq k$ and
       $I$ is smoothly spreadable, then $J$ is smoothly spreadable.
 \item[(2)] If  $J$ is smoothly spreadable, then $I$ is smoothly spreadable.
\end{enumerate}
\end{prop}

\begin{proof}
(1) Since $I$ is smoothly spreadable, according to Definition $1.7$, there exists a permutation 
$\tau:\{1,2,\ldots,nd\} \rightarrow \{1,2,\ldots,nd\}, 
\text{ with } \tau(j)-j \equiv 0 (\bmod\; n), \;(\forall)1\leq j\leq nd,$ 
such that the $K$-algebra isomorphism
 $$\Phi:T_{nd}\rightarrow T_{td},\;\Phi(x_j):=x_{\tau(j)},\;(\forall)1\leq j\leq nd,$$
satisfies $\Phi(\sigma^n(u)) = u^p$, for all $u\in G(I)$. For a monomial $u\in T_n$, we let $u^{(\ell)}$ be the monomial obtained 
from $u$ by replacing $x_{j_{\ell}},\ldots,x_n$ with $1$. Let
\begin{equation}\label{p12-0}
M_{\ell}:=\max\{\deg(u^{(\ell)})\;:\;u\in G(I),\;x_{j_{\ell}}|u\},\; m_{\ell}:=\min\{\deg(u^{(\ell)})\;:\;u\in G(I),\;x_{j_{\ell}}|u\}.
\end{equation}
Since, from hypothesis, $G(J)=G(I)\cup \{x_{j_1}^{d_1},\ldots,x_{j_k}^{d_k}\}$ and $d_{\ell}\geq M_{\ell}$, $(\forall)1\leq\ell\leq m$, it follows that
\begin{equation}\label{p12-1}
1\leq m_{\ell}\leq M_{\ell}\leq d_{\ell},\;(\forall)1\leq \ell \leq m. 
\end{equation}
Let $1\leq \ell \leq k$ and let $v_{\ell}:=x_{j_{\ell}}^{d_{\ell}}$. We have that
\begin{equation}\label{p12-2}
 \sigma^{n}(v_{\ell}) = v_{\ell}^{p} = x_{j_{\ell}}x_{j_{\ell}+n}\cdots x_{j_{\ell}+(d_{\ell}-1)n}.
\end{equation}
Note that $
\Phi|_{x_{\{j_{\ell}+m_{\ell}n},\ldots, x_{j_{\ell}+M_{\ell}n}\}}: 
\{ {x_{j_{\ell}+m_{\ell}n},\ldots, x_{j_{\ell}+M_{\ell}n}} \} \rightarrow \{x_{j_{\ell}},\cdots, x_{j_{\ell}+(M_{\ell}-m_{\ell})n}\}
$
is a bijection. We define the $K$-algebra isomorphism $\overline \Phi:T_{nd}\rightarrow T_{nd}$ by
\begin{eqnarray*}
& \overline \Phi(x_{j_{\ell}+sn}):=x_{j_{\ell}+(M_{\ell}+s)n},&\;(\forall)1\leq \ell \leq m,\;0\leq s\leq m_{\ell}-1,\\
& \overline \Phi(x_{j_{\ell}+sn}):=x_{j_{\ell}+sn},&\;(\forall)1\leq \ell \leq m,\;M_{\ell}+1\leq s\leq d_{\ell},\\
& \overline \Phi(x_j):=\Phi(x_j),&\text{in the other cases}.
\end{eqnarray*}
From \eqref{p12-0}, \eqref{p12-1}, \eqref{p12-2} and the definition of $\overline \Phi$ it follows that 
$$\overline \Phi(\sigma^{n}(v_{\ell}))=\sigma^{n}(v_{\ell})=v_{\ell}^{p}, (\forall)1\leq \ell\leq r,\; 
\overline \Phi(\sigma^{n}(u))=\Phi(\sigma^{n}(u))=u^{p}, \;(\forall)u\in G(I),$$
hence we get the required conclusion.

(2) It follows immediately from the fact that $G(I)\subset G(J)$.
\end{proof}

The following result gives another method of constructing smoothly spreadable ideals.

\begin{prop}
Let $d'>d \geq 1$, $n'>n\geq 1$ and $m,m'\geq 1$ be some integers. Let $\mathcal M=\{u_1,\ldots,u_m\}\subset T_n$ and $\mathcal M'=\{v_1,\ldots,v_{m'}\}\subset T_{n'}$ be two sets of monomials 
such that
\begin{enumerate}
 \item $\deg(u_i)=d,\;(\forall)1\leq i\leq m$.
 \item $\deg(v_i)\leq d'-d,\;(\forall)1\leq i\leq m'$. 
 \item $\supp(v_i)\subset \{x_{n+1},\ldots,x_{n'}\},\;(\forall)1\leq i\leq m'$.
 \item $\mathcal M$ and $\mathcal M'$ are smoothly spreadable.
\end{enumerate} 
Then the set $\{u_iv_k\;: 1\leq i\leq m,\;1\leq k\leq m' \}$ is smoothly spreadable. 
In particular, the ideal 
$$I:=(u_1,\ldots,u_m)\cdot (v_1,\ldots,v_m) = (u_1,\ldots,u_m)\cap (v_1,\ldots,v_m) \subset T_{n'}$$ 
is smoothly spreadable.
\end{prop}

\begin{proof}
From Definition $1.7$, there exist two permutations \small
\begin{eqnarray*}
& \tau:\{1,2,\ldots,nd\}\rightarrow \{1,2,\ldots,nd\},\; \tau(j)\equiv j(\bmod\; n),\;(\forall)1\leq j\leq nd,\\
& \tau':\{1,2,\ldots,n'(d'-d)\} \rightarrow \{1,2,\ldots,n'(d'-d)\},\; \tau'(j)\equiv j(\bmod\; n'),\;(\forall)1\leq j\leq n'(d'-d),
\end{eqnarray*} \normalsize
such that the $K$-algebra isomorphisms 
\begin{eqnarray*}
& \Phi:T_{nd}\rightarrow T_{nd}, \;\Phi(x_j)=x_{\tau(j)},\;1\leq j\leq nd \\
& \Phi':T_{n'(d'-d)}\rightarrow T_{n'(d'-d)}, \;\Phi(x_j)=x_{\tau'(j)},\;1\leq j\leq n'(d'-d),
\end{eqnarray*}
satisfy
\begin{equation}\label{p13-0}
\Phi(\sigma^n(u))=u^p,\;(\forall)u\in\mathcal M \text{ and } \Phi'(\sigma^{n'}(v))=v^{p'},\;(\forall)v\in\mathcal M'.
\end{equation}
We define $\overline \tau:\{1,\ldots,n'd'\} \rightarrow \{1,\ldots,n'd'\}$, by
\begin{equation}\label{p13-1}
  \overline \tau(j) =\begin{cases} \varphi(\tau(\psi(j,n')),n),\;& j\equiv 1,\ldots,n (\bmod \; n'),j\leq n'd \\ 
                                    \tau'(j),\;& j\equiv n+1,\ldots,n' (\bmod \; n'), j\leq n'(d'-d) \\
                                    j,\;& \text{ otherwise} \end{cases},
\end{equation}
where 
\begin{eqnarray*}
 & \psi(j,n'):=\floor*{\frac{j-1}{n'}}\cdot (n-n') +j,\;(\forall)1\leq j\leq n'd',\\
 & \varphi(j,n):=\floor*{\frac{j-1}{n}}\cdot (n'-n) +j,\;(\forall)1\leq j\leq nd. 
\end{eqnarray*}
Note that $\tau$ is a permutation on $\{1,2,\ldots,n'd'\}$. We consider the $K$-algebra isomorphism
$$ \overline \Phi:T_{n'd'} \rightarrow T_{n'd'},\; \overline \Phi (x_j):=x_{\overline \tau(j)},\;(\forall)1\leq j\leq n'd'.$$
Let $u\in \mathcal M$ and $v\in \mathcal M'$. We have that $u=x_1^{a_1}\cdots x_n^{a_n}$ with $a_1+\cdots+a_n=d$ 
and $v=x_{n+1}^{a_{n+1}}\cdots x_{n'}^{a_{n'}}$ with $a_{n+1}+\cdots+a_{n'}\leq d'-d$. It follows that                           
\begin{align*}
& \sigma^n(u)=x_1x_{n+1}\cdots x_{(a_1-1)n+1} \cdots x_{(a_1+\cdots+a_{n-1})n+n} \cdots x_{(a_1+\cdots+a_n-1)n+n},& \\
& u^p = x_1x_{n+1}\cdots x_{(a_1-1)n+1} \cdots x_{n}x_{2n} \cdots x_{a_nn},& \\
& \sigma^{n'}(u)=x_1x_{n'+1}\cdots x_{(a_1-1)n'+1} \cdots x_{(a_1+\cdots+a_{n-1})n'+n} \cdots x_{(a_1+\cdots+a_n-1)n'+n},& \\
& u^{p'}=x_1x_{n'+1}\cdots x_{(a_1-1)n'+1} \cdots x_{n}x_{n'+n} \cdots x_{(a_n-1)n'+n},& \\
& \sigma^{n'}(v)=x_{n+1}x_{n'+n+1}\cdots x_{(a_{n+1}-1)n'+n+1} \cdots x_{(a_{n+1}+\cdots+a_{n'-1}+1)n'} \cdots x_{(a_{n+1}+\cdots+a_{n'})n'},& \\
& v^{p'} = x_{n+1}x_{n'+n+1}\cdots x_{(a_{n+1}-1)n'+n+1} \cdots x_{n'}x_{2n'} \cdots x_{a_{n'}n'},
\end{align*}
therefore
$$ \sigma^{n'}(uv)= x_1x_{n+1}\cdots x_{(a_1-1)n+1} \cdots x_{(a_1+\cdots+a_{n-1})n+n} \cdots x_{(a_1+\cdots+a_n-1)n+n} \cdot x_{dn'+n+1}\cdot$$
\begin{equation}\label{p13-x}
\cdot x_{(d+1)n'+n+1}\cdots x_{(d+a_{n+1}-1)n'+n+1} \cdots x_{(d+a_{n+1}+\cdots+a_{n'-1}+1)n'} \cdots x_{(d+a_{n+1}+\cdots+a_{n'})n'} 
\end{equation}
From \eqref{p13-0} it follows that
\begin{eqnarray*}
& \tau(\{1,n+1,\cdots,(a_1-1)n+1, \ldots, (a_1+\cdots+a_{n-1})n+n,\ldots,(a_1+\cdots+a_n-1)n+n\}) \\
& = \{1,n+1,\ldots,(a_1-1)n+1,\ldots,n,2n,\ldots,a_n n\},
\end{eqnarray*}
hence, by \eqref{p13-1} it follows that
$$\overline \tau(\{1,n'+1,\ldots,(a_1-1)n'+1, \ldots, (a_1+\cdots+a_{n-1})n'+n,\ldots,(a_1+\cdots+a_n-1)n'+n\}) = $$
\begin{equation}\label{p13-2}
 = \{1,n'+1,\ldots,(a_1-1)n'+1,\ldots,n,n'+n,\ldots,(a_n-1)n'+ n\}.
\end{equation}
Similarly, from \eqref{p13-0} and \eqref{p13-1} it follows that 
$$
\overline \tau(\{dn'+n+1,(d+1)n'+n+1,\ldots,(d+a_{n+1}-1)n'+n+1, \ldots,(d+a_{n+1}+\cdots+a_{n'-1}+1)n',\ldots,$$
\begin{equation}\label{p13-3}
, (d+a_{n+1}+\cdots+a_{n'})n' \}) = \{n+1,n'+n+1,\ldots,(a_{n+1}-1)n'+n+1,\ldots,n',2n',\ldots,a_{n'}n'\}.
\end{equation}
From \eqref{p13-x}, \eqref{p13-2} and \eqref{p13-3} it follows that 
$$ \overline \Phi(\sigma^{n'}(uv)) = (uv)^{p'}.$$
Hence $\{u_iv_k\;: 1\leq i\leq m,\;1\leq k\leq m' \}$ is smoothly spreadable. 
The last assertion follows from the fact that $G(I\cdot J)\subset \{u_iv_k\;: 1\leq i\leq m,\;1\leq k\leq m' \}$.
\end{proof}

\begin{exm}
\emph{
(1) Let $I:=(x_1^2x_2,x_1x_2^2,x_2^3)\subset T_2$ and $J:=(x_3^2,x_3x_4^2)\subset T_4$. According to Proposition $1.13$ the ideals $I$
and $J\cap K[x_3,x_4]$ are smoothly spreadable. From Proposition $1.15$ it follows that the ideal
$I \cdot J = (x_1^2x_2x_3^2,\; x_1x_2^2x_3^2,\;x_2^3x_3^2,\;x_1^2x_2x_3x_4^2,\;x_1x_2^2x_3x_4^2,\;x_2^3x_3x_4^2)\subset T_4$
is smoothly spreadable.}

\emph{
(2) Let $I=(x_1x_2,x_2^2)\subset T_2$. According to Proposition $1.13$, the ideal $I$ is smoothly spreadable.
    Let $J=(x_1x_2x_3,x_2^2x_3^2)\subset T_2$. Let $\mathcal M=G(I)$ and $\mathcal M'=\{x_3,x_3^2\}$. Obviously, $\mathcal M'$
		is smoothly spreadable. By Proposition $1.15$, the set 
		$\mathcal G=\{x_1x_2x_3,x_1x_2x_3^2,x_2^2x_3,x_2^2x_3^2 \}$ is smoothly spreadable. Since $G(J)\subset \mathcal G$, it follows 
		that $J$ is smoothly spreadable.}
		
		\emph{		
		Let $L:=J+(x_1^2,x_2^3,x_3^4)\subset T_3$. Let $u_1=x_1x_2x_3$, $u_2=x_2^2x_3^2$, $d_1=2$, $d_2=3$ and $d_3=4$. 		
		Note that $G(L)=\{u_1,u_2,x_1^{d_1},x_2^{d_2},x_3^{d_3}\}$. With the notations from Proposition $1.14$, we have that
		$$u_1^{(1)}=x_1,\;u_2^{(1)}=1,\;u_1^{(2)}=x_1x_2,\;u_2^{(2)}=x_2^2,\;u_1^{(3)}=u_1,\;u_2^{(3)}=u_2.$$
		One can easily check that 
		$d_{\ell}\geq \max\{\deg(u_1^{(\ell)}),\deg(u_2^{\ell})\},\;(\forall)1\leq \ell\leq 3$,
		hence, from Proposition $1.14$ it follows that $L$ is smoothly spreadable}
\end{exm}

\begin{lema}
 Let $u,v\in T_n$ be two monomials, $d$ an integer with $\max\{\deg(u),\deg(v)\}\leq d$, $u=x_1^{a_1}\cdots x_n^{a_n}$ and $v=x_1^{b_1}\cdots x_n^{b_n}$.
 If $\{u,v\}$ is smoothly spreadable, then for any $1\leq j\leq n$: \small{
$$ \min\{a_j,b_j\} = |\{(a_1+\cdots+a_{j-1})n+j,(a_1+\cdots+a_{j-1}+1)n+j,\ldots,(a_1+\cdots+a_{j-1}+a_j-1)n+j\}\cap $$
$$ \cap \{(b_1+\cdots+b_{j-1})n+j,(b_1+\cdots+b_{j-1}+1)n+j,\ldots,(b_1+\cdots+b_{j-1}+b_j-1)n+j\}|. $$}
\end{lema}

\begin{proof}
From Definition $1.7$ there exists a permutation $\tau:\{1,2,\ldots,nd\}\rightarrow \{1,2,\ldots,nd\}$
such that the $K$-algebra isomorphism
$$\Phi:T_{nd}\rightarrow T_{nd},\; \Phi(x_j)=x_{\tau(j)},\;(\forall)1\leq j\leq nd,$$
satisfies the conditions $\Phi(\sigma^n(u))=u^p$ and $\Phi(\sigma^n(v))=v^p$. 
Since
\begin{eqnarray*}
& \sigma^n(u)=x_1x_{n+1}\cdots x_{(a_1-1)n+1}\cdots x_{(a_1+\cdots+a_{n-1}+1)n} \cdots x_{(a_1+\cdots+a_{n})n},\\
& \sigma^n(v)=x_1x_{n+1}\cdots x_{(b_1-1)n+1}\cdots x_{(b_1+\cdots+b_{n-1}+1)n} \cdots x_{(b_1+\cdots+b_{n})n},\\
& u^p = x_1x_{n+1}\cdots x_{(a_1-1)n+1} \cdots x_{n}x_{2n}\cdots x_{a_nn},\\
& v^p = x_1x_{n+1}\cdots x_{(b_1-1)n+1} \cdots x_{n}x_{2n}\cdots x_{b_nn},
\end{eqnarray*}
it follows that \small{
\begin{eqnarray*}
& \tau(\{(a_1+\cdots+a_{j-1})n+j,(a_1+\cdots+a_{j-1}+1)n+j,\ldots,(a_1+\cdots+a_{j}-1)n+j\})=\\
& = \{j,j+n,\ldots,j+(a_j-1)n\},\;(\forall)1\leq j\leq n,\\
& \tau(\{(b_1+\cdots+b_{j-1})n+j,(b_1+\cdots+b_{j-1}+1)n+j,\ldots,(b_1+\cdots+b_{j}-1)n+j\})=\\
& =\{j,j+n,\ldots,j+(b_j-1)n\},\;(\forall)1\leq j\leq n.
\end{eqnarray*}}
Since $\tau$ is a bijection, we get the required conclusion.
\end{proof}

\begin{teor}
Let $\mathcal M=\{u_1,\ldots,u_m\}\subset T_n$ be a set of monomials.  We write
$$u_i:=\prod_{j=1}^n x_j^{a_{i,j}}, \;(\forall) 1\leq i\leq m.$$
The following are equivalent:
\begin{enumerate}
 \item[(1)] $\mathcal M$ is smoothly spreadable.
 \item[(2)] For any $1\leq i < \ell \leq m$ and $1\leq j\leq n$ we have that 
\footnotesize
$$ \min\{a_{i,j},a_{\ell j}\} = |\{(a_{i,1}+\cdots+a_{i,j-1})n+j,(a_{i,1}+\cdots+a_{i,j-1}+1)n+j,\ldots,(a_{i,1}+\cdots+a_{i,j}-1)n+j\} $$
$$\cap\{(a_{\ell,1}+\cdots+a_{\ell,j-1})n+j,(a_{\ell,1}+\cdots+a_{\ell ,j-1}+1)n+j,\ldots,(a_{\ell,1}+\cdots+a_{\ell ,j-1}+a_{\ell,j}-1)n+j\}|.$$
\normalsize
\end{enumerate}
\end{teor}

\begin{proof}
 $(1)\Rightarrow (2)$ If $m=|\mathcal M|=1$ then there is nothing to prove. Assume $m\geq 2$ and let $1\leq i < \ell \leq m$. Since $\mathcal M$ is smoothly spreadable,
then the set $\{u_i,u_{\ell }\}$ is also smoothly spreadable. Now apply Lemma $1.17$.

$(2)\Rightarrow (1)$ We use induction on $n\geq 1$. If $n=1$, then there is nothing to prove. Assume $n\geq 2$. 
Let $$u'_i = \prod_{j=1}^{n-1}x_j^{a_{i,j}} \in T_{n-1},\;(\forall)1\leq i\leq m.$$
Let $1\leq i<\ell\leq m$ and $1\leq j\leq n-1$. The hypothesis (2) implies
\begin{eqnarray*}
 & \min\{a_{i,j},a_{\ell,j}\} = |\{(a_{i,1}+\cdots+a_{i,j-1})(n-1)+j,(a_{i,1}+\cdots+a_{i,j-1}+1)(n-1)+j,\ldots, \\
& ,(a_{i,1}+\cdots+a_{i,j}-1)(n-1)+j\} \cap\{(a_{\ell,1}+\cdots+a_{\ell ,j-1})(n-1)+j, \\ 
& (a_{\ell,1}+\cdots+a_{\ell ,j-1}+1)(n-1)+j,\ldots,(a_{\ell,1}+\cdots+a_{\ell ,j-1}+a_{\ell,j}-1)(n-1)+j\}|. 
\end{eqnarray*}
Hence, by induction hypothesis, $\mathcal M':=\{u'_1,\ldots,u'_m\}$ is smoothly spreadable. 

Let $d=\max_{i=1}^n deg(u_i)$.
Let $\tau':\{1,2,\ldots,(n-1)d\} \rightarrow \{1,2,\ldots,(n-1)d\}$ be a permutation and
$$\Phi':T_{(n-1)d}\rightarrow T_{(n-1)d}, \; \Phi'(x_j)=x_{\tau'(j)},\;(\forall)1\leq j\leq (n-1)d,$$ 
such that 
$$\Phi'(\sigma^{(n-1)}(u'_i)) = (u'_i)^{p'},\;(\forall)1\leq i\leq m,$$ 
where $()^{p'}$ is the polarization in $T_{n-1}$. Let 
$$d:=\max_{1\leq i\leq m} \deg(u_i) \text{ and } d':=\max_{1\leq i\leq m} \deg(u'_i).$$
Without loss of generality, we can assume that there exists $1\leq p\leq m$ such that $x_n|u_1$, $\ldots$ ,$x_n|u_p$, $x_n\nmid u_{p+1},\ldots,x_n\nmid u_m$ and 
$\deg(u'_1)\geq \deg(u'_2)\geq \cdots \geq \deg(u'_p)$. Let 
$$\alpha_i:=a_{i,1}+\cdots+a_{i,n-1},\; 1\leq i\leq p.$$
Let $1\leq i<\ell \leq p$ and assume that $\deg(u'_i)>\deg(u'_{\ell})$. We prove that $\deg(u_i)\leq \deg(u_{\ell})$. Indeed, if $\deg(u_i)> \deg(u_{\ell})$, then 
$\alpha_i > \alpha_{\ell} \text{ and }\alpha_i+a_{i,n}>\alpha_{\ell}+a_{\ell,n}$,
therefore we have
$$|\{\alpha_i n, (\alpha_i+1)n, \ldots, (\alpha_i+a_{i,n}-1)n\} \cap \{\alpha_{\ell} n, (\alpha_{\ell}+1)n, \ldots, 
(\alpha_{\ell}+a_{\ell,n}-1)n\}| <\min\{a_{i,n},a_{\ell,n}\},$$
a contradiction. By reordering $u_1,\ldots,u_p$ we can assume that $\deg(u_1)\leq \ldots \leq \deg(u_p)$.

We use a similar argument with the one from the proof of Proposition $1.13$.
We define the permutation $\lambda:\{1,2,\ldots,d\} \rightarrow \{1,2,\ldots,d\}$ by
\begin{eqnarray*}
& \lambda(j)=a_{p,n} + j, & 1\leq j\leq \alpha_p \\
& \lambda(\alpha_1+j)=j, & 1\leq j\leq a_{1,n} \\
& \lambda(\alpha_i+j) = a_{i-1,n} + j, & 1\leq j\leq \alpha_{i-1}-\alpha_i,\;2\leq i\leq p \\
& \lambda(\alpha_i+a_{i,n}+j) = a_{i,n} + \alpha_i - \alpha_{i+1} + j, & 1\leq j\leq \deg(u_{i+1})-\deg(u_i),\;1\leq i\leq p-1 \\
& \lambda(j)=j,& j > \deg(u_p) = \alpha_p+a_{p,n}. 
\end{eqnarray*}
We define the permutation $\tau:\{1,2,\ldots,nd\}\rightarrow \{1,2,\ldots,nd\}$ by
$$ \tau(j):= \begin{cases} \varphi(\tau'(\psi(j,n)),n-1),& j\equiv 1,\ldots,n-1 (\bmod n) \\ j,& j\equiv 0 (\bmod n) \end{cases},$$
where 
$$\psi(j,n):=-\floor*{\frac{j-1}{n}} +j \text{ and } \varphi(j,n-1):=\floor*{\frac{j-1}{n-1}} +j.$$ 
We consider the $K$-algebra isomorphism
$$\Phi:T_{nd}\rightarrow T_{nd},\; \Phi(x_j):=x_{\tau(j)},\;(\forall)1\leq j\leq nd.$$
As in the proof of Proposition $1.13$, by straightforward computations, we get 
$$\Phi(\sigma^n(u_i))=u_i^{(p)},\; (\forall)1\leq i\leq m,$$ hence we are done.
\end{proof}

\begin{exm}
\emph{
(1) Let $u_1=x_1x_2^2x_3^2$, $u_2=x_2^3x_3^3$ and $I=(u_1,u_2)\subset T_3$. With the notations from Theorem $1.18$, we have
$a_{1,1}=1,\;a_{1,2}=2,\;a_{1,3}=2,\;a_{2,1}=0,\;a_{2,2}=3,\;a_{2,3}=3$. Also
\begin{eqnarray*}
& 0=\min\{a_{1,1},a_{2,1}\} = |\{1\}\cap \emptyset|,\\
& 2=\min\{a_{1,2},a_{2,2}\} = |\{5,8\}\cap \{1,5,8\}| ,\\
& 2=\min\{a_{1,3},a_{2,3}\} = |\{12,15\}\cap \{12,15,18\}|, 
\end{eqnarray*}
hence $I$ is smoothly spreadable.}

\emph{
(2) Let $u'_1=x_1x_2x_3^2$, $u'_2=x_2^3x_3^3$ and $I'=(u_1,u_2)\subset T_3$. With the notations from Theorem $1.18$, we have
$a_{1,1}=1,\;a_{1,2}=1,\;a_{1,3}=2,\;a_{2,1}=0,\;a_{2,2}=3,\;a_{2,3}=3$. Since
$$ 2=\min\{a_{1,3},a_{2,3}\} \neq |\{9,12\}\cap \{12,15,18\}|=1$$
it follows that $I$ is not smoothly spreadable.
}
\end{exm}

\section{Lcm-lattices and spreading of monomial ideals}

A \emph{lattice} $L$ is a partially ordered set $(L,\geq)$ such that, for any $a,b\in L$, there is
a unique greatest lower bound $a\wedge b$ called the \emph{meet} of $a$ and $b$, and there is a
unique least upper bound $a\vee b$ called the \emph{join} of $a$ and $b$. 
% In this paper we only consider finite lattices, so in the remainder of the paper all lattices are assumed to be finite. 
A finite lattice $L$ has an unique minimal element $\hat 0$ and a unique maximal element $\hat 1$. 
An \emph{atom} is an element in $L$ that covers the minimal element $\hat 0$. A lattice is \emph{atomistic} if every element
can be written as a join of atoms.

Let $L$ and $L'$ be two finite lattices. A \emph{join preserving} map $\delta:L\rightarrow L'$ is a map with
$$\delta(a\vee b) = \delta(a)\vee \delta(b),\;(\forall)a,b\in L.$$
Note that every join-preserving map preserves the order.
A  \emph{meet preserving} map $\delta:L\rightarrow L'$ is a map with
$$\delta(a\wedge b) = \delta(a)\wedge \delta(b),\;(\forall)a,b\in L.$$
A map $\delta:L\rightarrow L'$ is a lattice isomorphism if $\delta$ is join preserving, meet preserving and bijective.
In the case that there exists an isomorphism between two lattices $L$ and $L'$ we say that $L$ and $L'$ are isomorphic
and we write $L\cong L'$.

Let $I\subset T_n$ be a monomial ideal. 
The \emph{lcm-lattice} of $I$, $L_I$ is defined as the set of all monomials that can be obtained as the 
least common multiple (lcm) of some minimal monomial generators of $I$, ordered by divisibility. More precisely,
if $G(I)=\{u_1,\ldots,u_m\}$ is the set of minimal monomial generators, then 
$$L_I=\{1,u_1,\ldots,u_m,\lcm(u_1,u_2),\ldots,\lcm(u_{m-1},u_m),\ldots,\lcm(u_1,\ldots,u_m)\}.$$
Note that $\hat 0=1$ and $\hat 1=\lcm(u_1,\ldots,u_m)$. The join of two elements in $L_I$ is their least common multiple.
See \cite{gash} for further details.
The lattice $L_I$ is atomistic and its atoms are exactly the elements of $G(I)$. 
Conversely, if $L$ is a finite atomistic lattice, Mapes had showed
that $L$ is isomorphic to a lcm-lattice associated to a monomial ideal, see \cite[Theorem 3.1]{map}.
An algorithm for computing all lcm-lattices of monomial ideals with a given number of minimal generators 
was given by Ichim, Katth\"an and Moyano-Fern\'andez in \cite[Section 3]{ichim3}.

Let $n'\geq n$ and $I'\subset T_{n'}$ be a monomial ideal such that $L_I\cong L_{I'}$.
Gasharov, Peeva and Welker proved in \cite[Theorem 3.3]{gash} that 
\begin{equation}\label{dep}
\depth(S/I) = \depth(S'/I')+n-n'.
\end{equation}
Ichim, Katth\"an and Moyano-Fern\'andez proved in \cite[Corollary 4.10]{ichim2} that 
\begin{equation}\label{sdep}
\sdepth(S/I) = \sdepth(S'/I')+n-n' \text{ and } \sdepth(I) = \sdepth(I')+n-n'.
\end{equation}

\begin{prop}
If $I$ is smoothly spreadable then $L_{I^{\sigma^n}}\cong L_{I}$
\end{prop}

\begin{proof}
From Definition $1.7$, it follows that $L_{I^{\sigma^n}}\cong L_{I^p}$.
On the other hand, $L_{I}\cong L_{I^p}$, see for instance \cite[Proposition 2.3]{gash}.
\end{proof}

\begin{obs} 
\emph{
Proposition $2.1$ shows that the monomial ideals which are smoothly spreadable have the 
property $L_{I} \cong L_{I^{\sigma^n}}$. However, the converse
is not true. Take for instance the ideal $I=(x_1^2x_2^2,x_2^3)\subset T_2$. From Proposition $1.13$ it follows that
$I$ is not smoothly spreadable. On the other hand, we have that
$I^{\sigma^2}=(x_1x_3x_6x_8,x_2x_4x_6)$. It is easy to check that $L_{I^{\sigma^2}}\cong L_{I}$.}

\emph{
There are also examples of monomial ideals $I\subset T_n$ with $L_I \ncong L_{I^{\sigma^n}}$. Take for instance the ideal $I=(x_1^4,x_1^2x_2,x_2^2)\subset T_2$. We have that $I^{\sigma^2}=(x_1x_3x_5x_7,x_1x_3x_6,x_2x_4)$ and
$$\lcm(x_1^4,x_2^2)=\lcm(x_1^4,x_1^2x_2,x_2^2) = x_1^4x_2^2 ,\; 
\lcm(x_1x_3x_5x_7,x_2x_4)\neq \lcm(x_1x_3x_5x_7,x_1x_3x_6,x_2x_4),$$
hence $L_I$ is not isomorphic to $L_{I^{\sigma^n}}$.}

\emph{
It would be interesting to describe the class of monomial ideals $I\subset T_n$ with the property 
$L_{I^{\sigma^n}}\cong L_I$, which extends the class of smoothly spreadable monomial ideals. 
For such ideals, according to \eqref{dep} and \eqref{sdep} it holds that
\begin{eqnarray*}
& \depth(T_{nd}/I^{\sigma^n})=\depth(T_n/I)+n(d-1), \\
& \sdepth(T_{nd}/I^{\sigma^n})=\sdepth(T_n/I)+n(d-1), \\
& \sdepth(I^{\sigma^n})=\sdepth(I)+n(d-1),
\end{eqnarray*}
where $d=\deg(I)$.}
\end{obs}

The following result is in the spirit of Lemma $1.1$.

\begin{lema}
Let $\{u_1,\ldots,u_m\}, \{u'_1,\ldots,u'_{m'}\} \subset T_n$ be two sets of monomials. Then
$$\lcm(\sigma^n(u_1),\ldots,\sigma^n(u_m))=\lcm(\sigma^n(u'_1),\ldots,\sigma^n(u'_{m'})) \Rightarrow 
\lcm(u_1,\ldots,u_m)=\lcm(u'_1,\ldots,u'_{m'}).$$
\end{lema}

\begin{proof}
We write
\begin{equation}\label{l22-1}
u_i=\prod_{j=1}^n x_j^{a_{i,j}},\;1\leq i\leq m,\;u'_{\ell}=\prod_{j=1}^n x_j^{a'_{\ell,j}},\;1\leq \ell \leq m'.
\end{equation}
We let
\begin{equation}\label{l22-2}
\alpha_j=\max\{a_{i,j}:\;1\leq i\leq m\}, \; \alpha'_j=\max\{a'_{\ell,j}:\;1\leq \ell\leq m'\}.
\end{equation}
From \eqref{l22-1} and \eqref{l22-2} it follows that
$$ \lcm(u_1,\ldots,u_m)=\prod_{j=1}^n x_j^{\alpha_j},\;\lcm(u'_1,\ldots,u'_m)=\prod_{j=1}^n x_j^{\alpha'_j},$$
hence 
\begin{equation}\label{l22-3}
\lcm(u_1,\ldots,u_m) = \lcm(u'_1,\ldots,u'_{m'}) \Leftrightarrow \alpha_j=\alpha'_j,\;(\forall)1\leq j\leq n.
\end{equation}
For each $1\leq i\leq m$ and $1\leq \ell \leq m'$ we have that
\small
\begin{eqnarray}
& \sigma^n(u_i)=x_1x_{n+1}\cdots x_{(a_{i,1}-1)n+1} \cdots x_{(a_{i,1}+\cdots+a_{i,n-1}+1)n}\cdots x_{(a_{i,1}+\cdots+a_{i,n})n}\label{l22-4}\\
& \sigma^n(u'_{\ell})=x_1x_{n+1}\cdots x_{(a'_{\ell,1}-1)n+1} 
\cdots x_{(a'_{\ell,1}+\cdots+a'_{\ell,n-1}+1)n}\cdots x_{(a'_{\ell,1}+\cdots+a'_{\ell,n})n}.\label{l22-5}
\end{eqnarray}
\normalsize
For each $1\leq j\leq n$ we let
$$\beta_j = \max\{a_{i,1}+a_{i,2}+\cdots+a_{i,j}:\;1\leq i\leq m\},\; \beta'_j=\max\{a_{\ell,1}+a_{\ell,2}+\cdots+a_{\ell,j}:\;1\leq \ell \leq m\}$$
We denote $u:=\lcm(\sigma^n(u_1),\ldots,\sigma^n(u_m))$ and $u':=\lcm(\sigma^n(u'_1),\ldots,\sigma^n(u'_{m'})$.
From \eqref{l22-4} and \eqref{l22-5} it follows that
$$
 x_{(\beta_j-1)n+j}\in u, \; x_{tn+j}\notin u, (\forall)t\geq \beta_j,\;
 x_{(\beta'_j-1)n+j}\in u', \; x_{tn+j}\notin u', (\forall)t\geq \beta'_j. $$
Since $u=u'$, it follows that $\beta_j=\beta'_j$ for all $1\leq j\leq n$, hence
$\alpha_j=\alpha'_j$ for all $1\leq j\leq n$. The conclusion follows from \eqref{l22-3}.
\end{proof}

\begin{lema}
Let $I\subset T_n$ be a monomial ideal with $G(I)=\{u_1,\ldots,u_m\}$. The map
$$\delta: L_{I^{\sigma^n}} \rightarrow L_I,\; \delta(\lcm(\sigma^n(u_{i_1}),\ldots,\sigma^n(u_{i_k}))):=\lcm(u_{i_1},\ldots,u_{i_k}),$$
where $0\leq k\leq m$ and $1\leq i_1<\cdots<i_k\leq m$, is join preserving, surjective and satisfies $\delta(1)=1$.
\end{lema}

\begin{proof}
From Lemma $2.3$ it follows that the map $\delta$ is well defined. 

Let $1\leq i_1<\cdots<i_k \leq m$ and $1\leq \ell_1<\cdots<\ell_t\leq m$.
Let $u=\lcm(u_{i_1},\ldots,u_{i_k})$, $\bar u =\lcm(\sigma^{n}(u_{i_1}),\ldots,\sigma^n(u_{i_k})$,
$v=\lcm(u_{\ell_1},\ldots,u_{\ell_t})$ and $\bar v =\lcm(\sigma^{n}(u_{\ell_1}),\ldots,\sigma^n(u_{\ell_t})$.
We have that
$$\delta(\bar u \vee \bar v) = \delta(\lcm(\bar u,\bar v)) = \lcm(u,v) = u \vee v = \delta(\bar u)\vee \delta(\bar v),$$
hence $\delta$ is join preserving. It is obvious that $\delta$ is surjective and $\delta(1)=1$.
\end{proof}

We recall the main results from \cite{ichim2}. \newpage

\begin{teor}
Let $n,n'\geq 1$ be two integers and let $I\subset T_n$ and $I'\subset T_{n'}$ be two monomial ideals.
Assume there exists a join preserving surjective map $\delta:L_{I'}\rightarrow L_I$ with $\delta(1)=1$. Then:
\begin{enumerate}
\item[(1)] $\depth(T_{n'}/I')\leq \depth(T_n/I)+n'-n$.
\item[(2)] $\sdepth(T_{n'}/I')\leq \depth(T_n/I)+n'-n$.
\item[(3)] $\sdepth(I')\leq \sdepth(I)+n'-n$.
\end{enumerate}
\end{teor}

\begin{proof}
(1) is a reformulation of \cite[Theorem 4.9]{ichim2}.

(2) and (3) are particular cases of \cite[Theorem 4.5]{ichim2}.
\end{proof}

The main result of the second section is the following:

\begin{teor}
Let $I\subset T_n$ be a monomial ideal with $\deg(I)=d$ and let $I^{\sigma^n}\subset T_{nd}$. We have that
\begin{enumerate}
\item[(1)] $\depth(T_{nd}/I^{\sigma^n})\leq \depth(T_n/I)+n(d-1)$.
\item[(2)] $\sdepth(T_{nd}/I^{\sigma^n})\leq \sdepth(T_n/I)+n(d-1)$.
\item[(3)] $\sdepth(I^{\sigma^n})\leq \sdepth(I)+n(d-1)$.
\end{enumerate}
Moreover, the equalities hold if $L_{I^{\sigma^n}}\cong L_I$ (in particular, if $I$ is smoothly spreadable).
\end{teor}

\begin{proof}
It follows from Lemma $2.4$ and Theorem $2.5$. The last assertion follows from \eqref{dep} and \eqref{sdep},
see also Remark $2.2$.
\end{proof}

\noindent
\textbf{Aknowledgment}: I would like to express my gratitude to the referee for his valuable remarks and suggestions which helped to improve this paper.

{}

\vspace{2mm} \noindent {\footnotesize
\begin{minipage}[b]{15cm}
Mircea Cimpoea\c s, Simion Stoilow Institute of Mathematics, Research unit 5, P.O.Box 1-764,\\
Bucharest 014700, Romania, E-mail: mircea.cimpoeas@imar.ro
\end{minipage}}

\begin{thebibliography}{}
\bibitem{mircea} M.\ Cimpoea\c s, \emph{A class of square-free monomial ideals associated to two integer sequences}, Comm. Algebra \textbf{46, no. 3}, (2018), 1179-1187.
\bibitem{ene} V.\ Ene, J.\ Herzog, A.\ A.\ Qureshi, \emph{t-spread strongly stable monomial ideals}, ArXiv e-prints (2018), https://arxiv.org/pdf/1805.02368.pdf
\bibitem{gash} V.\ Gasharov, I.\ Peeva, V.\ Welker, \emph{The lcm-lattice in monomial resolutions}, Math. Res. Lett. \textbf{5-6} (1999), 521-532.
\bibitem{hehi} J.\ Herzog, T.\ Hibi, \emph{Monomial ideals}, Grad. Texts in Math. 260, Springer, London, (2010).
\bibitem{hvz}  J.\ Herzog, M.\ Vladoiu, X.\ Zheng, \emph{How to compute the Stanley depth of a monomial ideal}, Journal of Algebra \textbf{322(9)}, (2009), 3151-3169.
\bibitem{ichim0} B.\ Ichim, L.\ Katth\"an, J.\ J.\ Moyano-Fern\'andez, \emph{The behavior of Stanley depth under polarization}, J. Combin. Theory Ser. A \textbf{135}, (2015), 332-347. 
% \bibitem{ichim} B.\ Ichim, L.\ Katth\"an, J.\ J.\ Moyano-Fern\'andez, \emph{How to compute the Stanley depth of a module}, Math. Comp. \textbf{86}, (2017), 455-472.
\bibitem{ichim3} B.\ Ichim, L.\ Katth\"an, J.\ J.\ Moyano-Fern\'andez, \emph{LCM Lattices and Stanley Depth: A First Computational Approach},
Experimental Mathematics, vol \textbf{25(1)}, (2016), 46-53. 
\bibitem{ichim2} B.\ Ichim, L.\ Katth\"an, J.\ J.\ Moyano-Fern\'andez, \emph{Stanley depth and the lcm-lattice}, J. Combin. Theory Ser. A \textbf{150}, (2017), 295-322.
\bibitem{kalai} G.\ Kalai, \emph{Algebraic shifting. In: Computational commutative algebra and combinatorics}, T. Hibi, (ed.), Adv. Stud. Pure Math. 33, Math. Soc. Japan, Tokyo (2002).
\bibitem{map} S.\ Mapes, \emph{Finite atomic lattices and resolutions of monomial ideals}, J. Algebra \textbf{379} (2013), 259-276.
\bibitem{okazaki} R.\ Okazaki, K.\ Yanagawa, \emph{Alternative polarizations of Borel fixed ideal, 
 Eliahou-Kervaire type resolution and discrete Morse theory}, J. Algebraic Combin. \textbf{38, no. 2} (2013), 407-436.
\bibitem{yana} K. Yanagawa, Alternative polarizations of Borel fixed ideals, Nagoya. Math. J. 207 (2012), 79-93.
% \bibitem{welker} V.\ Gasharov, I.\ Peeva, V.\ Welker, \emph{The lcm-lattice in monomial resolutions}, Mathematical Research Letters \textbf{6}, (1999), 521-532.
\end{thebibliography}
\end{document}